\documentclass[draftcls,onecolumn,12pt]{IEEEtran}
\usepackage{amsmath}
\usepackage{amssymb}
\usepackage{amsthm}
\usepackage[mathscr]{eucal}
\usepackage{eqlist} 
\usepackage[final]{graphicx}
\usepackage[dvipsnames]{color}
\DeclareGraphicsExtensions{.pdf, .jpg}
\usepackage{graphicx}
\usepackage{mdwlist}
\usepackage{cite}
\newtheorem{thm}{Theorem}[section]
\newtheorem{lem}[thm]{Lemma}

\newtheorem{cor}[thm]{Corollary}
\theoremstyle{definition}
\newtheorem{defn}[thm]{Definition}
\theoremstyle{remark}

\newtheorem{remark}[thm]{Remark}
\usepackage{multirow}
\usepackage{subfigure}
\usepackage{amsfonts}

\begin{document}
\title{Collaborative Network Formation in Spatial Oligopolies}
\author{Shaun Lichter, Terry Friesz, and Christopher Griffin}
\maketitle
\begin{abstract}
Recently, it has been shown that networks with an arbitrary degree sequence may be a stable solution to a network formation game.  Further, in recent years there has been a rise in the number of firms participating in collaborative efforts.  In this paper, we show conditions under which a graph with an arbitrary degree sequence is admitted as a stable firm collaboration graph.
\end{abstract}
\section{Introduction}
Recently there has been a rise in the number of firms participating in collaborative efforts.  Goyal and Joshi \cite{goyal2003} present a model of firm collaboration in aspatial oligopolies and in this paper we extend this model to spatial oligopolies.  We investigate the impact of the spatial economy on the collaboration network and the impact of the collaboration network on the spatial economy.
Since the 1960s the number of firm's participating in collaborative agreements has increased significantly \cite{hagedoorn1990,hagedoorn1993,hagedoorn1994,hagedoorn1996,hagedoorn2000,hagedoorn2002}.  This collaboration takes various forms, one of which is research and development (R \& D) that often consists of sharing resources such as equipment, laboratory space, office space, as well as engineers and scientists through separate R \& D subcompanies.  This collaboration has become very popular within industries that are R \& D intensive.  Hagedoorn shows that the number of collaborations has increased since the 1960s and rapidly increased in the 1980s \cite{hagedoorn2002}.  Firms in R \& D intensive industries are enabled to flexibly ally themselves to further their business.  Nonetheless, the existence of these collaborations is counterintuitive because firms should not want to share R \& D results or expenditures because it is the foundation of their future products.  As a result of this contradiction, this collaboration has spurred a host of literature \cite{hagedoorn1990,hagedoorn1993,hagedoorn1994,hagedoorn1996,hagedoorn2000,hagedoorn2002}.
Goyal and Joshi present a model of horizontal firm collaboration in oligopolies, where firms compete in the market after choosing collaborators \cite{goyal2003}.  The motivation behind this model is an examination of the incentives for collaboration and the interaction of these incentives with market competition.  Firms are able to lower production costs by committing some resources to a pair-wise collaboration effort.  A particular collaboration network is formed as a result of the collection of pairwise collaborations.  For each collaboration network, each firm has a particular production cost which effects the market competition that occurs over this collaboration network.  Hence, the oligopoly induces an allocation of value over the set of firms for a given collaboration network.
\section{Collaboration Networks and Collaborative Oligopolies}
In this section we present an introduction to collaboration networks and collaborative oligopolies.  We modify the notational conventions from the common notation in this this body of literature \cite{jackson1996,dutta1997,jackson2003} in order to better accomodate the spatial variables needed in later sections of this paper.  Let $N=\{1,2, \ldots n\}$ be the set of nodes in a graph, which will represent players or a group of players. The set of links in the graph is a set of pairs of nodes (subsets of $N$ of size two).  A graph $g$ is a set of links (set of subsets of $N$ of size two) and $g^{N}$ is the complete set of all links.  The set $G$ is the set of all graphs over the nodes $N$, that is, $G=\{g : g \subset g^{N} \}$.
The value of a graph $g$ is the total value produced by agents in the graph; we denote the value of a graph as the function $h:G \rightarrow \mathbb{R}$ and the set of of all such value functions as $H$. An allocation rule $Y:H \times G \rightarrow \mathbb{R}^{N}$ distributes the value $h(g)$ among the agents in $g$.  Denote the value allocated to agent $i$ as $Y_{i}(h,g)$.  Since, the allocation rule must distribute the value of the network to all players, it must be \textit{balanced}; i.e., $\sum_{i} Y_{i}(h,g)=h(g)$ for all $(h,g) \in H \times G$.  The allocation rule governs how the value is distributed and thus makes a significant contribution to the model.
Jackson and Wolinksy use \textit{pairwise stability} to model stable networks without the use of noncooperative games \cite{jackson1996}.
\begin{defn}
\label{definition: stability}
A network $g$ with value function $h$ and allocation rule $Y$ is pairwise \textit{stable} if (and only if):
\begin{enumerate*}
\item for all $ij \in g$, $Y_{i}(h,g) \geq Y_{i}(h,g - ij)$ and
\item for all $ij \not\in g$, if $Y_{i}(h,g+ij) > Y_{i}(h,g)$, then $Y_{j}(h,g+ij) < Y_{j}(h,g)$
\end{enumerate*}
\end{defn}
Pairwise stability implies that in a stable network, for each link that exists, (1) both players must benefit from it and (2) if a link can provide benefit to both players, then it in fact must exist.  Jackson notes that pairwise stability may be too weak because it does not allow groups of players to add or delete links, only pairs of players \cite{jackson2003}.  Deletion of multiple links simultaneously has been considered in \cite{belleflamme2004}.
We present an application of the network formation game to firm collaboration in spatial oligopolies, which is an extension to the firm collaboration presented by Goyal and Joshi in \cite{goyal2003}.

\subsection{General Collaborative Oligopoly Model}
Consider $n$ firms that compete in an oligopoly who may collaborate with any of the other $n-1$ firms. Firm $i$ produces a quantity $q_i$.  Denote $\mathbf{q}=(q_{1},q_{2},\ldots q_{n})$ as the vector of quantity production across all firms and $\mathbf{q}_{-i}=(q_{1},\ldots , q_{i-1},q_{i+1},\ldots q_{n})$ as the vector containing production quantities for all firms, but firm $i$.  Collaboration among firms affects the marginal cost of production. Thus a particular (collaboration) graph $g$ induces a marginal cost of firm $i$ under collaboration graph $g$ of $c_{i}(q_{i}|g)$.


We consider marginal cost functions of the form (\ref{eq: marginal cost 1}) where the marginal cost $c_{i}(q_{i}|g)$ for firm $i$ is a function of $q_{i}$, the quantity produced by firm $i$, and $\eta_{i}(g)$, the degree of firm $i$ in graph $g$.
\begin{equation}
\label{eq: marginal cost 1}
c_{i}(q_{i}|g)=f_{i}(q_{i},\eta_{i}(g))
\end{equation}
Here, $f_{i} \in C^{1}(\Omega_{i})$ where $q_{i} \in \Omega_{i}$ and $\Omega_{i}$ is defined as the feasible region for firm $i$.
\begin{displaymath}
\Omega_{i}= \left\{q_{i}:0 \leq q_i \right\}
\end{displaymath}

Given a network $g$, there is an induced set of costs which, along with the demand functions, produces a set of profit functions for each firm, $Y_{i}(g)$ (the allocation of payoff for player $i$).  These profit functions then induce a Nash equilibrium of production, which provides the precise allocation rule (i.e., profit) for each firm on the graph.  The stability of the collaboration network can then be analyzed using the definition of stability \ref{definition: stability}.

Denote the market marginal price function as $P(q_{1},q_{2},\ldots q_{n})$.  In this paper, we consider a market marginal price function (dependent on quantity produced) given by
\begin{equation}\label{eq: market demand}
P(q_{1},q_{2},\ldots q_{n}) = \alpha - \sum_{i \in N} q_{i}
\end{equation}
This can also be denoted as $P(Q)=\alpha-Q$ where $Q = \sum_{i \in N} q_i$
The profit for Player $i$ is:
\begin{equation}
Y_i(q_{i}|\mathbf{q}_{-i},g) =
\left(\alpha - \sum_{i \in N} q_{i}\right) q_i -
c(q_{i}|g)q_i
\label{eqn:Profit1}
\end{equation}
Given collaboration graph $g$, firm $i$ will solve the problem
\begin{align}
\max \;\; &Y_{i}(q_{i}|\mathbf{q}^{\ast}_{-i},g) \nonumber \\
\text{s.t.} \;\;& q_{i} \in \Omega_i
\end{align}
where $\mathbf{q}^{\ast}_{-i}$ is composed of the optimal production quantities for all firms, but $i$.  The gradient of the objective for firm $i$:
\begin{displaymath}
\nabla_{q_{i}} Y_{i}(q_{i}|\mathbf{q}^{\ast}_{-i},g)=
P(Q)-f_{i}(q_{i},\eta_{i}(g))-q_{i}-q_{i}\frac{\partial f_{i}}{\partial q_{i}} \\
\end{displaymath}
Each firm $i$ will solve an equivalent variational inequality by finding $q^{\ast}_{i} \in \Omega_{i}$ such that:
\begin{equation}
\langle  \nabla_{q_{i}} Y_{i}(q_{i}|\mathbf{q}^{\ast}_{-i},g),q_{i}-q^{\ast}_{i} \rangle \geq 0
\end{equation}
where $\langle \cdot , \cdot \rangle$ denotes a dot product.  In this case:
\begin{equation}
\langle P(Q)-f_{i}(q_{i},\eta_{i}(g))-q_{i}-q_{i}\frac{\partial f_{i}}{\partial q_{i}} ,q_{i}-q^{\ast}_{i} \rangle \geq 0
\end{equation}

The equilibrium for this oligopoly can be found by solving the variational inequality defined as finding $\mathbf{q}^{\ast} \in \Omega$ such that
\begin{equation}
\langle  \nabla_{\mathbf{q}} Y(\mathbf{q}|\mathbf{q}^{\ast},g),\mathbf{q}-\mathbf{q}^{\ast} \rangle \geq 0
\end{equation}
where
\begin{equation}
[\nabla_{\mathbf{q}} Y(\mathbf{q}|\mathbf{q}^{\ast},g)]_{i}=P(Q)-f_{i}(q_{i},\eta_{i}(g))-q_{i}-q_{i}\frac{\partial f_{i}}{\partial q_{i}}
\end{equation}
It is difficult to analytically determine which collaboration graphs will be stable because the oligopoly equilibriums are solutions to a variational inequality.  One could empirically find stable graphs, but instead we seek to find subcases of the model for which we can find analytical results.

\subsection{Previous Results on Network Stability in Aspatial Oligopoly}
In Goyal and Joshi \cite{goyal2003}, it is assumed that the marginal cost of firm $i$ linearly decreases with the number of collaborators for firm $i$:
\begin{equation}\label{eq: marginal cost}
c_{i}(g)=\gamma_{0}-\gamma \eta_{i}(g)
\end{equation}
where, as before, $\eta_{i}(g)$ is the number of links for firm $i$ and $\gamma_{0}$ is the marginal cost of production when a firm has no links.  Notice that $\gamma_{0}$ is constant for all firms.
One example that Goyal and Joshi \cite{goyal2003} study is that of a homogenous product oligopoly.  With the market marginal price function (\ref{eq: market demand}) and marginal cost (\ref{eq: marginal cost}), the resulting profit to Player $i$ is:
\begin{equation}
Y_i(g) =
\left(\alpha - \sum_{i \in N} q_{i}\right) q_i -
\left(\gamma_{0}-\gamma \eta_{i}(g)\right)q_i =
(\alpha - \gamma_0)q_i + \left(\sum_{i \in N} q_{i}\right)q_i -
\left(-\gamma \eta_{i}(g)\right)q_i
\label{eqn:Profit1}
\end{equation}
Goyal and Joshi show that with marginal cost (\ref{eq: marginal cost}) and market demand (\ref{eq: market demand}), the complete network is the unique stable network \cite{goyal2003}.
\subsection{Results of Nonlinear Cost on Stability}
In this section we review the results from \cite{LichterGriffinFriesz2011}, where we show the effect a nonlinear variation on the marginal cost function has on the stability of collaboration structures. In particular, we show that with cost functions of a particular form, the collaborative oligopoly will result in a stable collaboration graph with an arbitrary degree sequence.  We consider a marginal cost function:
\begin{equation}
c_i(g) = \gamma_0 + f_i(\eta_i(g))
\label{eqn:MarginalCostf}
\end{equation}
where $f_{i}$ is some function $f_{i}:\mathbb{R} \rightarrow \mathbb{R}$.

\begin{lem}
\label{lem: qi}
Suppose we have an oligopoly consisting of $n$ firms in which collaboration is defined by the graph $g$ and the profit function (allocation rule) for Firm $i$ in that oligopoly is given by:
\begin{equation}
Y_{i}(g) = (\alpha - \gamma_0)q_{i}(g) - \left(\sum_{j \in N} q_{j}\right)q_i(g) - f_{i}(\eta_{i}(g))q_i(g)
\label{eqn: profit oligopoly}
\end{equation}
then the quantity produced for firm $i$ is:
\begin{equation}
\label{eq: quantity produced app1}
q_{i}(g) = \frac{\alpha - \gamma_{0}-n f_{i}(\eta_{i}(g))+ \sum_{j \neq i}f_{j}(\eta_{j}(g)) }{n+1}
\end{equation}
\end{lem}
\begin{proof} From \cite{Tir88}, for any oligopoly with profit function of the form:
\begin{equation}
\pi_i(q) = aq_i - \left(\sum_{j \in N} q_{j}\right)q_i - b_iq_i
\end{equation}
The resulting Cournot equilibrium point on quantities is:
\begin{equation}
q_i = \frac{a - nb_i + \sum_{j \neq i}b_j}{n+1}
\label{eqn:CournotEq}
\end{equation}
In our case, we have:
\begin{gather*}
a = \alpha - \gamma_0\\
b_i = f_{i}(\eta_{i}(g)) \quad \forall i
\end{gather*}
Substituting these definitions into Expression (\ref{eqn:CournotEq}) yields Expression (\ref{eq: quantity produced app1}). This completes the proof.
\end{proof}
\begin{remark}
It is worth noting that when for each firm $i$, $b_i = -\gamma \eta_{i}(g)$ then the cost function (\ref{eq: marginal cost}) and induced equilibrium quantity (\ref{eqn:CournotEq}) is equivalent to that used in Goyal and Joshi.
\end{remark}

\begin{cor}
\label{cor: qi nonnegative}
Suppose that $f$ is a nonnegative ($f (\eta) \geq 0$) convex function that has a minimum at $0$.  Further, suppose $f_{i}(\eta_{i}(g))=f(\eta_{i}(g)-k_{i})$ where $k_{i} \in \{0,1,\ldots ,n-1\}$.
If the parameters $\alpha$ and $\gamma_{0}$ and the function $f$ are such that:
\begin{equation}
\label{eq: qi nonnegative}
\alpha - \gamma_{0}-n \max \{ f(n-1),f(1-n) \}-\frac{1}{2}(n-1)\max \{ f(1)-f(0),f(-1)-f(0) \} > 0
\end{equation}
and $n \geq 2$, then the Cournot equilibrium quantities (\ref{eq: quantity produced app1}) are nonnegative for all firms and for all collaboration graphs and the following inequalities hold:
\begin{align}
2q_{i}(g)-\frac{n-1}{n+1} [f(1)-f(0)] &>0 \label{eq: ineq 1} \\
2q_{i}(g)-\frac{n-1}{n+1}[f(-1)-f(0)] &>0  \label{eq: ineq 2}
\end{align}
\end{cor}

\begin{proof}
Since $n \geq 2$ and $f$ is convex and has a minimum at $0$, this implies that $\frac{n-1}{n+1} [f(1)-f(0)]$ and $\frac{n-1}{n+1} [f(-1)-f(0)]$ are non-negative.  If (\ref{eq: ineq 1}) and  (\ref{eq: ineq 2}) hold, then $q_{i}(g)$ is non-negative and hence it suffices to only show that (\ref{eq: ineq 1}) and  (\ref{eq: ineq 2}) are implied by (\ref{eq: qi nonnegative}).

For all $i$, function $f_{i}$ is a convex function of the degree of node $i$ in the graph $g$; the degree of node $i$ must take an integer value between $0$ and $n-1$, which due to the convexity of $f_{i}$ and the fact that $k_{i} \in \{0,1,\ldots n-1\}$ implies that the maximum of $f_{i}$ is equivalent to $\max \{f_{i}(0),f_{i}(n-1) \}$ and is less than $\max\{f(n-1),f(-n+1)\}$.  That is,
\begin{equation*}
f_{i}(\eta_{i}(g)) \leq \max \{f_{i}(\eta_{i}(g)) \}=\max \{f_{i}(0),f_{i}(n-1) \} \leq \max\{f(n-1),f(-n+1)\}
\end{equation*}
This means that (\ref{eq: qi nonnegative}) implies:
\begin{equation}
\label{eq: ineq 3}
\alpha - \gamma_{0}-n f_{i}(\eta_{i}(g))-\frac{1}{2}(n-1)\max \{f(1)-f(0),f(-1)-f(0) \} > 0 \quad \forall i
\end{equation}
Since, all $f_{i}(\eta_{i}(g)) \geq 0$, we may add $\sum_{j \neq i}f_{j}(\eta_{j}(g))$ to the left side of  (\ref{eq: ineq 3}) without changing the inequality:
\begin{equation}
\label{eq: ineq 4}
\alpha - \gamma_{0}-n f_{i}(\eta_{i}(g))+\sum_{j \neq i}f_{j}(\eta_{j}(g))-\frac{1}{2}(n-1)\max(f(1)-f(0),f(-1)-f(0)) > 0 \quad \forall i
\end{equation}
Dividing by $n+1$ yields:
\begin{equation}
\label{eq: ineq 5}
\frac{\alpha - \gamma_{0}-n f_{i}(\eta_{i}(g))+\sum_{j \neq i}f_{j}(\eta_{j}(g))}{n+1}-\frac{1}{2}\frac{n-1}{n+1}\max(f(1)-f(0),f(-1)-f(0)) > 0 \quad \forall i
\end{equation}
From Lemma \ref{lem: qi} this simplifies to:
\begin{equation}
\label{eq: ineq 6}
q_{i}(g)-\frac{1}{2}\frac{n-1}{n+1}\max(f(1)-f(0),f(-1)-f(0)) > 0 \quad \forall i
\end{equation}
Multiplying through by two yields:
\begin{align*}
2q_{i}(g)>\frac{n-1}{n+1}\max(f(1)-f(0),f(-1)-f(0)) &>\frac{n-1}{n+1}[f(1)-f(0)]  \quad \forall i\\
2q_{i}(g)>\frac{n-1}{n+1}\max(f(1)-f(0),f(-1)-f(0)) &>\frac{n-1}{n+1}[f(-1)-f(0)]  \quad \forall i
\end{align*}
Now (\ref{eq: ineq 1}) and  (\ref{eq: ineq 2}) immediately follow.
\end{proof}
\begin{remark}
This essentially means that the steeper a function $f$ around zero and on the interval $(-n+1,n-1)$, the greater the quantity $\alpha-\gamma_0$ is needed to ensure the theorem proved later in this section.  It is worth pointing out that this bound may often not be tight (i.e., the inequalities may hold true and production quantities may be positive even when the condition is not met).
\end{remark}

\begin{thm}
\label{thm: stable assym k distn graph}
Suppose that $f$ is a nonnegative ($f (\eta) \geq 0$) convex function that has a minimum at $0$.  Further, suppose $f_{i}(\eta_{i}(g))=f(\eta_{i}(g)-k_{i})$.  Define the change in $f$ as $\triangle^{-} f_{i}(k_{i})=f_{i}(k_{i}-1)-f_{i}(k_{i})=f(-1)-f(0)=\triangle^{-} f(0)$ and $\triangle^{+} f_{i}(k_{i})=f_{i}(k_{i}+1)-f_{i}(k_{i})=f(1)-f(0)=\triangle^{+} f(0)$.  Suppose $n \geq 2$ firms compete in an oligopoly with market demand $P=\alpha - \sum_{i \in N}q_{i}$ and marginal costs $c_i(g) = \gamma_0 + f_i(\eta_i(g))$.  If the parameters $\alpha$ and $\gamma_{0}$ and the functions $f_i$ obey condition (\ref{eq: qi nonnegative}), then the equivalence class of graphs $[g]_{\eta}$ such that $\eta_{i}(g)=k_{i}$ is an equivalence class of stable collaboration graphs.
\end{thm}
\begin{proof}
Let $g$ be a graph in the equivalence class of graphs $[g]_{\eta}$, that is, $g$ has a degree sequence such that $\eta_{i}(g)=k_{i}$ for all firms $i$.  Consider a firm $i$ who may consider dropping its link with node $j$.  If node $i$ drops its link with node $j$ leading to graph $g-ij$, then $\eta_{i}(g-ij)=k_{i}-1$ and $\eta_{j}(g-ij)=k_{j}-1$, while $\eta_{r}(g-ij)=k_{r}$ for $r \not \in \{i,j\}$.  
Using Lemma \ref{lem: qi}
\begin{equation}
q_{i} = \frac{\alpha - \gamma_{0}-n f_{i}(\eta_{i}(g))+ \sum_{j \neq i \in N}f_{j}(\eta_{j}(g)) }{n+1}
\end{equation}
Calculate:
\begin{align*}
q_{i}(g-ij) &= q_{i}(g)- \triangle^{-} f_{i}(k_{i}) \left( \frac{n}{n+1} \right)+ \triangle^{-} f_{j}(k_{j}) \left( \frac{1}{n+1}  \right)\\
q_{j}(g-ij) &= q_{j}(g)- \triangle^{-} f_{j}(k_{j}) \left( \frac{n}{n+1} \right)+ \triangle^{-} f_{i}(k_{i}) \left( \frac{1}{n+1}  \right)\\
q_{r}(g-ij) &= q_{r}(g)+ \triangle^{-} f_{i}(k_{i}) \left( \frac{1}{n+1} \right)+ \triangle^{-} f_{j} (k_{j})\left( \frac{1}{n+1}  \right)
\end{align*}
It then follows that
\begin{align*}
Q(g-ij)&=Q(g)- \left( \frac{1}{n+1} \right) (\triangle^{-} f_{i}(k_{i})+\triangle^{-} f_{j}(k_{j}))\\
P(g-ij)&=P(g)+\left( \frac{1}{n+1} \right)(\triangle^{-} f_{i}(k_{i})+\triangle^{-} f_{j}(k_{j}))  \\
c_{i}(g-ij)&=c_{i}(g)+\triangle^{-} f_{i}(k_{i})
\end{align*}
Now, we can calculate $Y_{i}(g-ij)$ in terms of $Y_{i}(g)$:
\begin{align*}
Y_{i}(g-ij)&=q_{i}(g-ij)[P(g-ij)-c_{i}(g-ij)] \\
   &= Y_{i}(g)+ q_{i}(g) \left( \frac{2}{n+1} \right) [  \triangle^{-} f_{j}(k_{j}) - n \triangle^{-} f_{i}(k_{i})]
   +\left( \frac{[  \triangle^{-} f_{j}(k_{j}) - n \triangle^{-} f_{i}(k_{i})]}{n+1} \right)^2
\end{align*}
Since $f_{i}(\eta_{i}(g))=f(\eta_{i}(g)-k_{i})$ this implies that $\triangle^{-} f_{i}(k_{i})=\triangle^{-} f_{j}(k_{j})$ and we obtain (\ref{eq: profit change 1A}) and then (\ref{eq: profit change 2A}) and (\ref{eq: profit change 3A}) through algebraic manipulation.  Finally, by the assumptions of the theorem and condition (\ref{eq: qi nonnegative}) each of the quantities $\triangle^{-} f_{i}(k_{i})$, $\frac{n-1}{n+1}$, and $2 q_{i}(g) -  \frac{n-1}{n+1} \triangle^{-} f_{i}(k_{i})$ are nonnegative implying (\ref{eq: profit change 4A}).
\begin{align}
Y_{i}(g-ij)-Y_{i}(g)&=2 q_{i}(g)\triangle^{-} f_{i}(k_{i}) \left( \frac{1-n}{n+1} \right)+(\triangle^{-} f_{i}(k_{i}))^{2} \left( \frac{1-n}{n+1} \right)^2 \label{eq: profit change 1A}\\
&=  \triangle^{-} f_{i}(k_{i}) \left(  \frac{1-n}{n+1}   \right) \left(2 q_{i}(g) + \frac{1-n}{n+1} \triangle^{-} f_{i}(k_{i})    \label{eq: profit change 2A} \right)\\
&= - \triangle^{-} f_{i}(k_{i}) \left(  \frac{n-1}{n+1}   \right) \left(2 q_{i}(g) -  \frac{n-1}{n+1} \triangle^{-} f_{i}(k_{i})   \label{eq: profit change 3A} \right)\\
&<0 \label{eq: profit change 4A}
\end{align}
This implies that if firm $i$ attempts to drop link $ij$, then $Y_{i}(g) > Y_{i}(g-ij)$ and thus firm $i$ decreases its profit.  The same will be true for firm $j$.  Hence, no firm has an incentive to drop a link from graph $g$.
Now, we will consider the case where firm $i$ attempts to add a link to the graph $g$, giving $g+ij$ under the assumption that the link $ij$ does not exist in graph $g$.  This analysis will follow closely the analysis for $g-ij$.  First note that $\eta_{i}(g)=k_{i}$ for all firms $i$ and $\eta_{i}(g+ij)=k_{i}+1$ and $\eta_{j}(g+ij)=k_{j}+1$, while $\eta_{r}(g+ij)=k_{r}$ for $r \not \in \{i,j\}$.  We define $\triangle^{+} f_{i}(k_{i})$ as $\triangle^{+} f_{i}(k_{i})=f_{i}(k+1)-f_{i}(k)$; note the subtle difference from the definition of $\triangle^{-} f_{i}(k_{i})$.  Again using Lemma \ref{lem: qi}, we calculate the production quantity for each node in graph $g+ij$:
\begin{align*}
q_{i}(g+ij) &= q_{i}(g)- \triangle^{+} f_{i}(k_{i}) \left( \frac{n}{n+1} \right)+ \triangle^{+} f_{j}(k_{j}) \left( \frac{1}{n+1}  \right)\\
q_{j}(g+ij) &= q_{j}(g)- \triangle^{+} f_{j}(k_{j}) \left( \frac{n}{n+1} \right)+ \triangle^{+} f_{i}(k_{i}) \left( \frac{1}{n+1}  \right)\\
q_{r}(g+ij) &= q_{r}(g)+ \triangle^{+} f_{i}(k_{i}) \left( \frac{1}{n+1} \right)+ \triangle^{+} f_{j}(k_{j}) \left( \frac{1}{n+1}  \right)
\end{align*}
We can then calculate the corresponding total production quantity $Q$, the market price $P$ and marginal costs for each player for the graph $g+ij$:
\begin{align*}
Q(g+ij)&=Q(g)- \left( \frac{1}{n+1} \right) (\triangle^{+} f_{i}(k_{i})+\triangle^{+} f_{j}(k_{j}))\\
P(g+ij)&=P(g)+\left( \frac{1}{n+1} \right)(\triangle^{+} f_{i}(k_{i})+\triangle^{+} f_{j}(k_{j}))  \\
c_{i}(g+ij)&=c_{i}(g)+\triangle^{+} f_{i}(k_{i})
\end{align*}
Now, we can calculate $Y_{i}(g+ij)$ in terms of $Y_{i}(g)$:
\begin{align*}
Y_{i}(g+ij)&=q_{i}(g+ij)[P(g+ij)-c_{i}(g+ij)] \\
   &= Y_{i}(g)+ q_{i}(g) \left( \frac{2}{n+1} \right) [  \triangle^{+} f_{j}(k_{j}) - n \triangle^{+} f_{i}(k_{i})]
   +\left( \frac{[  \triangle^{+} f_{j}(k_{j}) - n \triangle^{+} f_{i}(k_{i})]}{n+1} \right)^2
\end{align*}
Since $f_{i}(\eta_{i}(g))=f(\eta_{i}(g)-k_{i})$ this implies that $\triangle^{+} f_{i}(k_{i})=\triangle^{+} f_{j}(k_{j})$ and we obtain (\ref{eq: profit change 5}) and then (\ref{eq: profit change 6}) and (\ref{eq: profit change 7}) through algebraic manipulation.  Finally, by the assumptions of the theorem and condition (\ref{eq: qi nonnegative}), each of the quantities $\triangle^{+} f_{i}(k_{i})$, $\frac{n-1}{n+1}$, and $2 q_{i}(g) -     \frac{n-1}{n+1} \triangle^{+} f_{i}(k_{i})$ are positive implying (\ref{eq: profit change 8}).
\begin{align}
Y_{i}(g+ij)-Y_{i}(g)&=2 q_{i}(g)\triangle^{+} f_{i}(k_{i}) \left( \frac{1-n}{n+1} \right)+(\triangle^{+} f_{i}(k_{i}))^{2} \left( \frac{1-n}{n+1} \right)^2 \label{eq: profit change 5}\\
&=  \triangle^{+} f_{i}(k_{i}) \left(  \frac{1-n}{n+1}   \right) \left(2 q_{i}(g) + \frac{1-n}{n+1} \triangle^{+} f_{i}(k_{i})    \label{eq: profit change 6} \right)\\
&= - \triangle^{+} f_{i}(k_{i}) \left(  \frac{n-1}{n+1}   \right) \left(2 q_{i}(g) -     \frac{n-1}{n+1} \triangle^{+} f_{i}(k_{i})   \label{eq: profit change 7} \right)\\
&<0 \label{eq: profit change 8}
\end{align}
This implies that if firm $i$ attempts to add a link $ij$, then $Y_{i}(g) > Y_{i}(g+ij)$ and the firm decreases its profit.  The same will be true for firm $j$.  Hence, no firm has an incentive to add a link to graph $g$. Since no firm has an incentive to add or drop a link to graph $g$, it is stable. This completes the proof.
\end{proof}

\section{Collaborative Spatial Oligopolies}
Spatial Oligopolies (Oligopolies on spatially separated markets) have been studied extensively \cite{Harker1986,dafermos1987,Miller1996,Nagurney2009}.  In this section we extend the collaborative oligopoly model of Goyal and Joshi by applying it to spatially separated markets and we extend the existing literature in spatial oligopolies by allowing firm collaboration.  We seek to find which graphs $g$ are stable collaboration graphs.  As in prior sections, $N=\{1,2, \ldots n\}$ will denote firms, which are nodes on the collaboration graph.  Alternatively, there is a spatial transport network with nodes denoted as $V=\{1,2, \ldots v\}$.  Consumer demand at transport node $l \in V$ for firm $i \in N$ is denoted as $d_{li}$ and the total demand at node $l$ is denoted as $D_{l}=\sum_{i}d_{li}$.
Denote the vector  $\mathbf{d}_{i}=[d_{li}]_{l \in V}$ as the demand vector for firm $i$ across all nodes.  The quantity produced by firm $i$ is again denoted as $q_{i}$. Noting that $q_{i} =\sum_{l \in V}{d_{li}}$, we can eliminate $q_{i}$ by formulating all expressions in terms of $\mathbf{d}_{i}=[d_{li}]_{l \in V}$.

The induced price at node $l$ is denoted as $P_{l}(D_{l})=P_{l}(d_{l1},d_{l2},\ldots,d_{ln})$.  The marginal production cost is $c_{i}(q_{i}|g)=f_{i}(q_{i},\eta_{i}(g))$ as before in (\ref{eq: marginal cost}) but is now denoted as $c_{i}(\sum_{l \in V}{d_{li}}|g)=f_{i}(\sum_{l \in V}{d_{li}},\eta_{i}(g))$.  However, now there is an additional marginal cost to ship a unit of quantity to node $l$ for firm $i$ denoted as $s_{li}$
\footnote{Each firm is not explicitly placed on the transport network, but its location may be implied through the $s_{li}$ values}.  Define $Y_{i}(\mathbf{d}_{i}|g)$ as the profit for firm $i$ with collaboration graph $g$:
\begin{equation*}
Y_{i}(\mathbf{d}_{i}|g)=\sum_{l\in V}{d_{li} \left[ p_{l}({D_{l}})-s_{li}-f_{i}\left(\sum_{l \in V}{d_{li}},\eta_{i}(g)\right) \right] }
\end{equation*}
Hence, the firm $i$ will solve the problem
\begin{align}
\max \;\; &Y_{i}(\mathbf{d}_{i}|g) \nonumber \\
\text{s.t.} \;\;& \mathbf{d}_{i} \in \Theta_{i}
\end{align}
where $\Theta_{i}=\{\mathbf{d}_{i}:\mathbf{d}_{i} \geq \mathbf{0}\}$.  We can calculate the gradient of the objective for firm $i$:
\begin{displaymath}
\nabla_{\mathbf{d}_{i}} Y_{i}(\mathbf{d}_{i})= \left[
P_{l}({D_{l}})-s_{li} - f_{i}\left(\sum_{l \in V}{d_{li}},\eta_{i}(g)\right)-d_{li}-d_{li} \frac{\partial f_{i}(\sum_{l \in V}{d_{li}},\eta_{i}(g))}{\partial d_{li}}
\right]_{l \in V}
\end{displaymath}
Each firm $i$ will solve the equivalent variational inequality by finding $\mathbf{d}^{\ast}_{i} \geq 0$ such that:
\begin{equation}
\langle \nabla_{\mathbf{d}_{i}} Y_{i}(\mathbf{d}^{\ast}_{i}|g),\mathbf{d}_{i} -\mathbf{d}^{\ast}_{i}   \rangle \geq 0
\end{equation}


We may now find an equilibrium to the spatial oligopoly for all firms by solving the single composed variational inequality.
Find $\mathbf{d}^{\ast} \geq 0$ such that:
\begin{equation}
\langle \nabla_{\mathbf{d}} Y(\mathbf{d}^{\ast}|g),\mathbf{d} -\mathbf{d}^{\ast}   \rangle \geq 0
\end{equation}
Where $\mathbf{d}=[\mathbf{d}_{i}]_{i \in N}$ and $Y(\mathbf{d}^{\ast}|g)=[Y_{i}(\mathbf{d}^{\ast}_{i}|g)]_{i \in N}$.

With such a spatial model, it again becomes difficult to analytically find stable graphs.  Stability is difficult to determine analytically because in order to determine if a link should exist, the value a node receives from the link must be contrasted from the value without the link.  This is difficult without using sensitivity analysis for variational inequalities.  Instead we seek to show a set of models that do yield analytical results.

\subsection{Nonlinear production costs in Spatial Collaborative Oligopoly}
Consider a marginal cost function:
\begin{equation}
c_i(g) = \gamma_0 + f_i(\eta_i(g))
\label{eqn:MarginalCostSpatial}
\end{equation}
where $f_{i}$ is some function $f_{i}:\mathbb{R} \rightarrow \mathbb{R}$.  The marginal cost to ship a unit of quantity to node $l$ for firm $i$ is again denoted as $s_{li}$.
Each firm maximizes its profit by solving its own nonlinear problem:
\begin{align}
\max &\sum_{l\in V}{d_{li}[P_{l}(d_{l1},d_{l2},\ldots,d_{ln})-s_{li} -\gamma_0 - f_i(\eta_i(g))]}
\nonumber \\
\text{s.t.}\;\; &  0 \leq d_{li} \nonumber \;\;\;\;\;\;\;\;\;\;\; \forall{i \in N},l \in V
\end{align}
\begin{remark}
This nonlinear program that each firm will solve has been decoupled, such that now at each transport node $l$, the firms participate in oligopolistic competition that is independent from the competition at each other node.  However, at each node, each firm has a different cost due to the variability of the shipment cost to that node for each firm.
\end{remark}
\begin{lem}
\label{lem: dif}
Suppose we have an oligopoly consisting of $n$ firms in which collaboration is defined by the graph $g$, the demand function at node $l$ is $P_{l}(d_{l1},d_{l2},\ldots,d_{ln}) = \alpha_l - \sum_{i \in N}{d_{li}}$, and the profit function (allocation rule) for Firm $i$ in that oligopoly is given by:
\begin{equation}
Y_{i}(g,d_{1i},d_{2i},\ldots,d_{li})=\sum_{l\in V}{d_{li}[ \alpha_l - \sum_{j \in N}{d_{lj}} -s_{li} -\gamma_0 - f_i(\eta_i(g))]}
\label{eqn: profit oligopoly}
\end{equation}
then the demand met at node $l$ by firm $i$ is:
\begin{equation}
\label{eq: quantity produced app2}
d_{li}=\frac{\alpha_l-\gamma_0 -n(s_{li}+ f_i(\eta_i(g)) ) + \sum_{j \neq i} [s_{lj}+ f_j(\eta_j(g))  ]    }{n+1}
\end{equation}
\end{lem}
\begin{proof}
The profit for firm $i$ can be rearranged:
\begin{align*}
Y_{i}(g,d_{1i},d_{2i},\ldots,d_{li})&=\sum_{l\in V}{d_{li}[ \alpha_l - \sum_{j \in N}{d_{lj}} -s_{li} -\gamma_0 - f_i(\eta_i(g))]}\\
&=\sum_{l\in V} (\alpha_l-\gamma_0  ){d_{li}}- \left(   \sum_{j \in N}{d_{lj}} \right){d_{li}}-  \left(s_{li}+ f_i(\eta_i(g)) \right)d_{li}
\end{align*}
From \cite{Tir88}, for any oligopoly with profit function of the form:
\begin{equation}
Y_i(q) = aq_i - \left(\sum_{j \in N} q_{j}\right)q_i - b_iq_i
\end{equation}
The resulting Cournot equilibrium point on quantities is:
\begin{equation}
q_i = \frac{a - nb_i + \sum_{j \neq i}b_j}{n+1}
\label{eqn:CournotEq2}
\end{equation}
In our case, we have an oligopoly at each location $l$ and quantities $d_{li}$ with parameters :
\begin{gather*}
a = \alpha_{l} - \gamma_0\\
b_i = s_{li}+f_{i}(\eta_{i}(g)) \quad \forall i
\end{gather*}
Substituting these definitions into Expression (\ref{eqn:CournotEq2}) yields Expression (\ref{eq: quantity produced app2}). This completes the proof.
\end{proof}

\begin{cor}
\label{cor: dli nonnegative new}
Suppose that $f$ is a nonnegative ($f (\eta) \geq 0$) convex function that has a minimum at $0$.  Further, suppose $f_{i}(\eta_{i}(g))=f(\eta_{i}(g)-k_{i})$ where $k_{i} \in \{0,1,\ldots ,n-1\}$.
If the function $f$ and parameters $\alpha$, $\gamma_{0}$, and $s$ are such that:
\begin{multline}
\label{eq: dli nonnegative new}
   \alpha - \gamma_{0}-n \left[ \max_{l \in V,i \in N}s_{li}+\max \{ f(n-1),f(1-n) \} \right]\\
   -\frac{1}{2}(n-1)\max \{ f(1)-f(0),f(-1)-f(0) \} > 0
\end{multline}
and $n \geq 2$, then the Cournot equilibrium quantities (\ref{eq: quantity produced app2}) are nonnegative for all firms at all locations and for all collaboration graphs and the following inequalities hold:
\begin{align}
2d_{li}(g)-\frac{n-1}{n+1} [f(1)-f(0)] &>0 \quad \forall i \in N, l \in V \label{eq: ineq 1 new} \\
2d_{li}(g)-\frac{n-1}{n+1}[f(-1)-f(0)] &>0 \quad \forall i \in N, l \in V  \label{eq: ineq 2 new}
\end{align}
\end{cor}

\begin{proof}
Since $n \geq 2$ and $f$ is convex and has a minimum at $0$, this implies that $\frac{n-1}{n+1} [f(1)-f(0)]$ and $\frac{n-1}{n+1} [f(-1)-f(0)]$ are non-negative.  If (\ref{eq: ineq 1 new}) and  (\ref{eq: ineq 2 new}) hold, then $d_{li}(g)$ is non-negative and hence it suffices to only show that (\ref{eq: ineq 1 new}) and  (\ref{eq: ineq 2 new}) are implied by (\ref{eq: dli nonnegative new}).

For all $i$, function $f_{i}$ is a convex function of the degree of node $i$ in the graph $g$; the degree of node $i$ must take an integer value between $0$ and $n-1$, which due to the convexity of $f_{i}$ and the fact that $k_{i} \in \{0,1,\ldots n-1\}$ implies that the maximum of $f_{i}$ is equivalent to $\max \{f_{i}(0),f_{i}(n-1) \}$ and is less than $\max\{f(n-1),f(-n+1)\}$.  That is,
\begin{equation*}
f_{i}(\eta_{i}(g)) \leq \max \{f_{i}(\eta_{i}(g)) \}=\max \{f_{i}(0),f_{i}(n-1) \} \leq \max\{f(n-1),f(-n+1)\} \quad \forall i \in N
\end{equation*}
Further, $s_{li} \leq \max_{l \in V,i \in N}s_{li}$.  This means that (\ref{eq: dli nonnegative new}) implies:
\begin{equation}
\label{eq: ineq 3 new}
   \alpha - \gamma_{0}-n \left[s_{li}+f_{i}(\eta_{i}(g)) \right]-\frac{1}{2}(n-1)\max \{f(1)-f(0),f(-1)-f(0) \} > 0 \quad \forall i \in N, l \in V
\end{equation}
Since, all $f_{i}(\eta_{i}(g)) \geq 0$ and all $s_{li}\geq 0$, we may add $\sum_{j \neq i}[s_{lj}+f_{j}(\eta_{j}(g))]$ to the left side of  (\ref{eq: ineq 3 new}) without changing the inequality:
\begin{multline}
\label{eq: ineq 4 new}
\alpha - \gamma_{0}-n \left[s_{li}+f_{i}(\eta_{i}(g)) \right]+\sum_{j \neq i}[s_{lj}+f_{j}(\eta_{j}(g))]\\
-\frac{1}{2}(n-1)\max \{f(1)-f(0),f(-1)-f(0) \} > 0 \quad \forall i \in N, l \in V
\end{multline}
Dividing by $n+1$ yields:
\begin{multline}
\label{eq: ineq 5}
\frac{\alpha - \gamma_{0}-n \left[s_{li}+f_{i}(\eta_{i}(g)) \right]+\sum_{j \neq i}[s_{lj}+f_{j}(\eta_{j}(g))]}{n+1} \\
-\frac{1}{2}\frac{n-1}{n+1}\max \{f(1)-f(0),f(-1)-f(0) \} > 0 \quad \forall i \in N, l \in V
\end{multline}
From Lemma \ref{lem: dif} this simplifies to:
\begin{equation}
\label{eq: ineq 6}
d_{li}(g)-\frac{1}{2}\frac{n-1}{n+1}\max \{f(1)-f(0),f(-1)-f(0)\} > 0 \quad \forall i \in N, l \in V
\end{equation}
Multiplying through by two yields:
\begin{align*}
2d_{li}(g)>\frac{n-1}{n+1}\max \{f(1)-f(0),f(-1)-f(0)\} &>\frac{n-1}{n+1}[f(1)-f(0)]  \quad \forall i \in N, l \in V\\
2d_{li}(g)>\frac{n-1}{n+1}\max\{f(1)-f(0),f(-1)-f(0)\} &>\frac{n-1}{n+1}[f(-1)-f(0)]  \quad \forall i \in N, l \in V
\end{align*}
Now (\ref{eq: ineq 1 new}) and  (\ref{eq: ineq 2 new}) immediately follow.
\end{proof}


\begin{remark}
It should be noted that this bound will often not be tight and hence demand quantities may be positive even when it is not met.
\end{remark}

Suppose that $f$ is a convex function that has a minimum at $0$.  Further, suppose $f_{i}(\eta_{i}(g))=f(\eta_{i}(g)-k_{i})$.  Define the change in $f$ as $\triangle^{-} f_{i}(k_{i})=f_{i}(k_{i}-1)-f_{i}(k_{i})=f(-1)-f(0)=\triangle^{-} f(0)$ and $\triangle^{+} f_{i}(k_{i})=f_{i}(k_{i}+1)-f_{i}(k_{i})=f(1)-f(0)=\triangle^{+} f(0)$.  Suppose $n \geq 2$ firms compete in an oligopoly with market demand $P=\alpha - \sum_{i \in N}q_{i}$ and marginal costs $c_i(g) = \gamma_0 + f_i(\eta_i(g))$.

If the parameters $\alpha$ and $\gamma_{0}$ and the functions $f_i$ obey condition (\ref{eq: qi nonnegative}), then the equivalence class of graphs $[g]_{\eta}$ such that $\eta_{i}(g)=k_{i}$ is an equivalence class of stable collaboration graphs.

The induced price at node $l$ is denoted as $P_{l}(D_{l})=P_{l}(d_{l1},d_{l2},\ldots,d_{ln})$.  The marginal production cost is $c_{i}(q_{i}|g)=f_{i}(q_{i},\eta_{i}(g))$ as before in (\ref{eq: marginal cost}) but is now denoted as $c_{i}(\sum_{l \in V}{d_{li}}|g)=f_{i}(\sum_{l \in V}{d_{li}},\eta_{i}(g))$.  However, now there is an additional marginal cost to ship a unit of quantity to node $l$ for firm $i$ denoted as $s_{li}$

Suppose we have an oligopoly consisting of $n$ firms in which collaboration is defined by the graph $g$, the demand function at node $l$ is $P_{l}(d_{l1},d_{l2},\ldots,d_{ln}) = \alpha_l - \sum_{i \in N}{d_{li}}$, and the profit function (allocation rule) for Firm $i$ in that oligopoly is given by:

\begin{thm}
\label{thm: stable assym k distn graph}
Suppose that $f$ is a nonnegative ($f (\eta) \geq 0$) convex function that has a minimum at $0$.  Further, suppose $f_{i}(\eta_{i}(g))=f(\eta_{i}(g)-k_{i})$.  Define the change in $f$ as $\triangle^{-} f_{i}(k_{i})=f_{i}(k_{i}-1)-f_{i}(k_{i})=\triangle^{-} f(0)=f(-1)-f(0)$ and $\triangle^{+} f_{i}(k_{i})=f_{i}(k_{i}+1)-f_{i}(k_{i})=\triangle^{+} f(0)=f(1)-f(0)$.  Suppose $n \geq 2$ firms compete in an oligopoly with market demand $P_{l}(d_{l1},d_{l2},\ldots,d_{ln}) = \alpha_l - \sum_{i \in N}{d_{li}}$, marginal production cost of $c_{i}(g)=\gamma_0+f_{i}(\eta_{i}(g))$, and marginal shipping cost of $s_{li}$.  If the parameters $\alpha$, $\gamma_{0}$, and $s$ as well as the function $f$ obey condition (\ref{eq: dli nonnegative new}), then the equivalence class of graphs $[g]_{\eta}$ such that $\eta_{i}(g)=k_{i}$ is an equivalence class of stable collaboration graphs.
\end{thm}
\begin{proof}
Let $g$ be a graph in the equivalence class of graphs $[g]_{\eta}$, that is, $g$ has a degree sequence such that $\eta_{i}(g)=k_{i}$ for all firms $i$.  Consider a firm $i$ who may consider dropping its link with node $j$.  If node $i$ drops its link with node $j$ leading to graph $g-ij$, then $\eta_{i}(g-ij)=k_{i}-1$ and $\eta_{j}(g-ij)=k_{j}-1$, while $\eta_{r}(g-ij)=k_{r}$ for $r \not \in \{i,j\}$.  
Using Lemma \ref{lem: dif}
\begin{equation}
d_{li}=\frac{\alpha_l-\gamma_0 -n(s_{li}+ f_i(\eta_i(g)) ) + \sum_{j \neq i} [s_{lj}+ f_j(\eta_j(g))  ]    }{n+1}
\end{equation}
Calculate:
\begin{align*}
d_{li}(g-ij) &= d_{li}(g)- \triangle^{-} f_{i}(k_{i}) \left( \frac{n}{n+1} \right)+ \triangle^{-} f_{j}(k_{j}) \left( \frac{1}{n+1}  \right)\\
d_{lj}(g-ij) &= d_{lj}(g)- \triangle^{-} f_{j}(k_{j}) \left( \frac{n}{n+1} \right)+ \triangle^{-} f_{i}(k_{i}) \left( \frac{1}{n+1}  \right)\\
d_{lr}(g-ij) &= d_{lr}(g)+ \triangle^{-} f_{i}(k_{i}) \left( \frac{1}{n+1} \right)+ \triangle^{-} f_{j} (k_{j})\left( \frac{1}{n+1}  \right)
\end{align*}
It then follows that
\begin{align*}
D_{l}(g-ij)&=D_{l}(g)- \left( \frac{1}{n+1} \right) (\triangle^{-} f_{i}(k_{i})+\triangle^{-} f_{j}(k_{j}))\\
P_{l}(g-ij)&=P_{l}(g)+\left( \frac{1}{n+1} \right)(\triangle^{-} f_{i}(k_{i})+\triangle^{-} f_{j}(k_{j}))  \\
c_{i}(g-ij)&=c_{i}(g)+\triangle^{-} f_{i}(k_{i})
\end{align*}
Define $Y_{i}(g)=\sum_{l \in V}{y_{li}(g)}$ where $y_{li}(g)=d_{li}(g)[P_{l}(g)-c_{i}(g)-s_{li}]$.  Now, we can calculate $y_{li}(g-ij)$ in terms of $y_{li}(g)$:
\begin{align*}
y_{li}(g-ij)&=d_{li}(g-ij)[P_{l}(g-ij)-c_{i}(g-ij)-s_{li}] \\
             &= y_{li}(g)+ d_{li}(g) \left( \frac{2}{n+1} \right) [  \triangle^{-} f_{j}(k_{j}) - n \triangle^{-} f_{i}(k_{i})]
             +\left( \frac{[  \triangle^{-} f_{j}(k_{j}) - n \triangle^{-} f_{i}(k_{i})]}{n+1} \right)^2
\end{align*}
Since $f_{i}(\eta_{i}(g))=f(\eta_{i}(g)-k_{i})$ this implies that $\triangle^{-} f_{i}(k_{i})=\triangle^{-} f_{j}(k_{j})$ yielding (\ref{eq: profit change 1}) and then (\ref{eq: profit change 2}) and (\ref{eq: profit change 3}) through algebraic manipulation.  Finally,  $\triangle^{-} f_{i}(k_{i})$ and $\frac{n-1}{n+1}$ are non-negative and by Corollary \ref{cor: dli nonnegative new}, the term $\left(2 d_{li}(g) -  \frac{n-1}{n+1} \triangle^{-} f_{i}(k_{i})  \right)$ is non-negative as well.  Hence, this implies (\ref{eq: profit change 4}).
\begin{align}
y_{li}(g-ij)-y_{li}(g)&=2 d_{li}(g)\triangle^{-} f_{i}(k_{i}) \left( \frac{1-n}{n+1} \right)+(\triangle^{-} f_{i}(k_{i}))^{2} \left( \frac{1-n}{n+1} \right)^2 \label{eq: profit change 1}\\
&=  \triangle^{-} f_{i}(k_{i}) \left(  \frac{1-n}{n+1}   \right) \left(2 d_{li}(g) + \frac{1-n}{n+1} \triangle^{-} f_{i}(k_{i})    \label{eq: profit change 2} \right)\\
&= - \triangle^{-} f_{i}(k_{i}) \left(  \frac{n-1}{n+1}   \right) \left(2 d_{li}(g) -  \frac{n-1}{n+1} \triangle^{-} f_{i}(k_{i})   \label{eq: profit change 3} \right)\\
&<0 \label{eq: profit change 4}
\end{align}
Since $y_{li}(g-ij)-y_{li}(g)<0$ for all $i$, we can sum over all transport nodes $l$, to see that this implies that node $i$ does not have an incentive to drop a link.
\begin{align*}
y_{li}(g-ij)-y_{li}(g)&<0\\
\sum_{l}y_{li}(g-ij)-\sum_{l} y_{li}(g)&<0\\
Y_{i}(g-ij)-Y_{i}(g)&<0
\end{align*}
This implies that if firm $i$ attempts to drop link $ij$, then $Y_{i}(g) > Y_{i}(g-ij)$ and thus firm $i$ decreases its profit.  The same will be true for firm $j$.  Hence, no firm has an incentive to drop a link from graph $g$.
Now, we will consider the case where firm $i$ attempts to add a link to the graph $g$, giving $g+ij$ under the assumption that the link $ij$ does not exist in graph $g$.  This analysis will follow closely the analysis for $g-ij$.  First note that $\eta_{i}(g)=k_{i}$ for all firms $i$ and $\eta_{i}(g+ij)=k_{i}+1$ and $\eta_{j}(g+ij)=k_{j}+1$, while $\eta_{r}(g+ij)=k_{r}$ for $r \not \in \{i,j\}$.  We define $\triangle^{+} f_{i}(k_{i})$ as $\triangle^{+} f_{i}(k_{i})=f_{i}(k+1)-f_{i}(k)$; note the subtle difference from the definition of $\triangle^{-} f_{i}(k_{i})$.  Again using Lemma \ref{lem: qi}, we calculate the production quantity for each node in graph $g+ij$:
\begin{align*}
d_{li}(g+ij) &= d_{li}(g)- \triangle^{+} f_{i}(k_{i}) \left( \frac{n}{n+1} \right)+ \triangle^{+} f_{j}(k_{j}) \left( \frac{1}{n+1}  \right)\\
d_{lj}(g+ij) &= d_{lj}(g)- \triangle^{+} f_{j}(k_{j}) \left( \frac{n}{n+1} \right)+ \triangle^{+} f_{i}(k_{i}) \left( \frac{1}{n+1}  \right)\\
d_{lr}(g+ij) &= d_{lr}(g)+ \triangle^{+} f_{i}(k_{i}) \left( \frac{1}{n+1} \right)+ \triangle^{+} f_{j}(k_{j}) \left( \frac{1}{n+1}  \right)
\end{align*}
We can then calculate the corresponding total production quantity $Q$, the market price $P$ and marginal costs for each player for the graph $g+ij$:
\begin{align*}
D_{l}(g+ij)&=D_{l}(g)- \left( \frac{1}{n+1} \right) (\triangle^{+} f_{i}(k_{i})+\triangle^{+} f_{j}(k_{j}))\\
P_{l}(g+ij)&=P_{l}(g)+\left( \frac{1}{n+1} \right)(\triangle^{+} f_{i}(k_{i})+\triangle^{+} f_{j}(k_{j}))  \\
c_{i}(g+ij)&=c_{i}(g)+\triangle^{+} f_{i}(k_{i})
\end{align*}
Now, we can calculate $y_{li}(g+ij)$ in terms of $y_{li}(g)$:
\begin{align*}
y_{li}(g+ij)&=d_{li}(g+ij)[P_{l}(g+ij)-c_{i}(g+ij)-s_{li}] \\
             &= y_{li}(g)+ d_{li}(g) \left( \frac{2}{n+1} \right) [  \triangle^{+} f_{j}(k_{j}) - n \triangle^{+} f_{i}(k_{i})]
             +\left( \frac{[  \triangle^{+} f_{j}(k_{j}) - n \triangle^{+} f_{i}(k_{i})]}{n+1} \right)^2
\end{align*}
Since $f_{i}(\eta_{i}(g))=f(\eta_{i}(g)-k_{i})$ this implies that $\triangle^{+} f_{i}(k_{i})=\triangle^{+} f_{j}(k_{j})$ yielding (\ref{eq: profit change 5b}) and then (\ref{eq: profit change 6b}) and (\ref{eq: profit change 7b}) through algebraic manipulation.  Finally,  $\triangle^{+} f_{i}(k_{i})$ and $\frac{n-1}{n+1}$ are non-negative and by Corollary \ref{cor: dli nonnegative new}, the term $\left(2 d_{li}(g) -  \frac{n-1}{n+1} \triangle^{+} f_{i}(k_{i})  \right)$ is non-negative as well.  Hence, this implies (\ref{eq: profit change 8b}).
\begin{align}
y_{li}(g+ij)-y_{li}(g)&=2 d_{li}(g)\triangle^{+} f_{i}(k_{i}) \left( \frac{1-n}{n+1} \right)+(\triangle^{+} f_{i}(k_{i}))^{2} \left( \frac{1-n}{n+1} \right)^2 \label{eq: profit change 5b}\\
&=  \triangle^{+} f_{i}(k_{i}) \left(  \frac{1-n}{n+1}   \right) \left(2 d_{li}(g) + \frac{1-n}{n+1} \triangle^{+} f_{i}(k_{i})    \label{eq: profit change 6b} \right)\\
&= - \triangle^{+} f_{i}(k_{i}) \left(  \frac{n-1}{n+1}   \right) \left(2 d_{li}(g) -  \frac{n-1}{n+1} \triangle^{+} f_{i}(k_{i})   \label{eq: profit change 7b} \right)\\
&<0 \label{eq: profit change 8b}
\end{align}
Since $y_{li}(g+ij)-y_{li}(g)<0$ for all $i$, we can sum over all transport nodes $l$, to see that this implies that node $i$ does not have an incentive to drop a link.
\begin{align*}
y_{li}(g+ij)-y_{li}(g)&<0\\
\sum_{l}y_{li}(g+ij)-\sum_{l} y_{li}(g)&<0\\
Y_{i}(g+ij)-Y_{i}(g)&<0
\end{align*}
This implies that if firm $i$ attempts to add a link $ij$, then $Y_{i}(g) > Y_{i}(g+ij)$ and the firm decreases its profit.  The same will be true for firm $j$.  Hence, no firm has an incentive to add a link to graph $g$. Since no firm has an incentive to add or drop a link to graph $g$, it is stable. This completes the proof.
\end{proof}

\section{Numerical Example}
\label{ex: assymetric oligopoly}
We present a numerical example of Theorem \ref{thm: stable assym k distn graph}. Let $N=5$ firms compete in an oligopoly with inverse demand function $P_{l}=103-\sum_{i}d_{li}$, fixed cost $\gamma_{0}=5$, shipping costs $s_{li}=1 \; \forall \; li$, and $f_{i}(\eta_{i}(g))=(\eta_{i}(g) - k_{i})^{2}+\psi$ where $\mathbf{k}=[2,3,4,3,2]^T$ and $\psi=2$.  We want to test the stability of a graph $g$ with $\eta_{i}(g)=k_{i}$ and $f_{i}(\eta_{i}(g))=(\eta_i(g)-k_{i})^{2}+\psi$ for each node $i$. Note that $f(\eta_{i}(g))=(\eta_{i}(g))^{2}+\psi$.  The following calculations will be need:

\begin{center}
\begin{tabular}{|c|c|c|c|c|}
  \hline
$\alpha$ & 103\\
\hline
$\gamma_{0}$ & 5 \\
\hline
$n$ & 5 \\
\hline
$\max_{l \in V,i \in N}s_{li}$ & 1\\
\hline
$f(n-1)=f(4)=4^2 + 2$ & 18\\
  \hline
$f(1-n)=f(-4)=(-4)^2+2$ & 18\\
  \hline
$f(1)=1^2+2$ & 3\\
  \hline
$f(-1)=(-1)^2+2$ & 3\\
  \hline
$f(0)=0^2+2$ & 2\\
  \hline
\end{tabular}
\end{center}

In order to invoke Corollary \ref{cor: dli nonnegative new}, we must ensure condition (\ref{eq: dli nonnegative new}) holds:
\begin{multline*}
\alpha - \gamma_{0}-n \left[ \max_{l \in V,i \in N}s_{li}+\max \{ f(n-1),f(1-n) \} \right]\\
-\frac{1}{2}(n-1)\max \{ f(1)-f(0),f(-1)-f(0) \} > 0
\end{multline*}
Plugging in the appropriate values:
\begin{align*}
   103-5-5 \cdot [1+\max(18,18) ]-\frac{1}{2}(4)\max \{ 3-2,3-2) \} &> 0 \\
   98-5 \cdot 19 -2 \cdot 1 &> 0 \\
   98-95 -2 &> 0 \\
   1 &> 0
\end{align*}

Condition (\ref{eq: dli nonnegative new}) is met for this set of parameters and function $f$.
\begin{figure}[ht]
\centering
\includegraphics[scale=0.40]{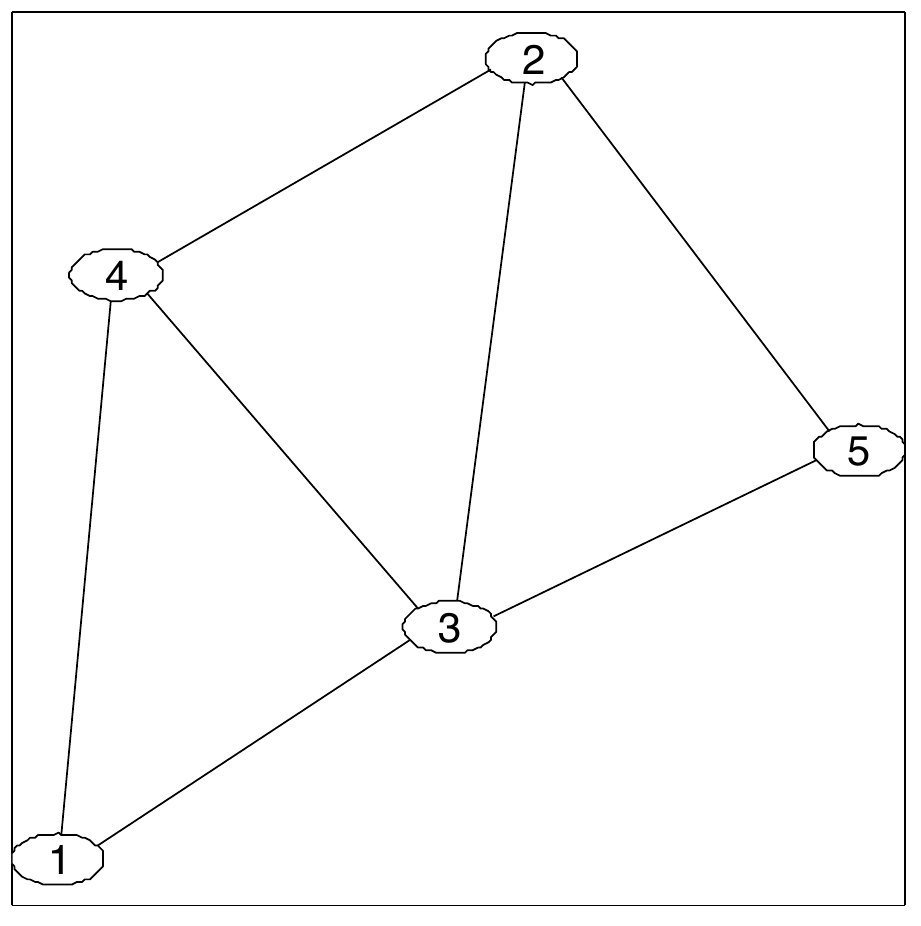}
\caption{Collaboration Network}
\label{fig:X1}
\end{figure}

\begin{figure}[ht]
\centering
\includegraphics[scale=0.40]{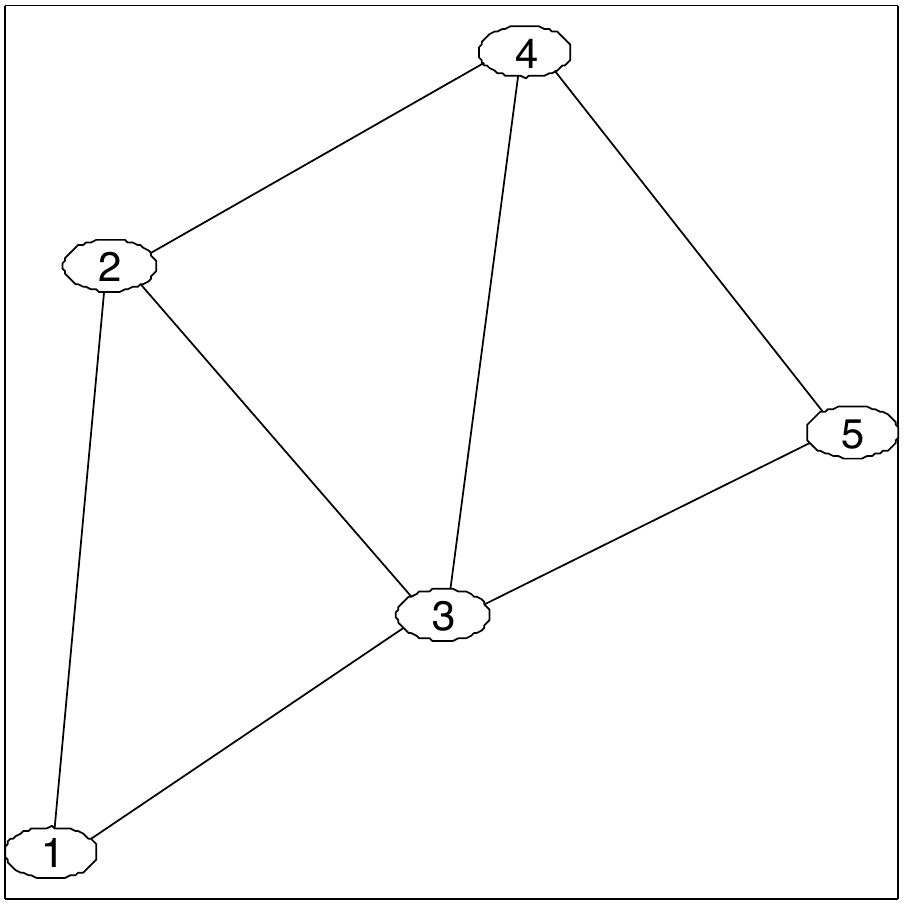}
\caption{Collaboration Network 2}
\label{fig:X1b}
\end{figure}


Two stable graphs, shown in Figure \ref{fig:X1} and Figure \ref{fig:X1b}, have a degree sequence equivalent to $\mathbf{k}$.



\section{Conclusion}
In this paper we bridge the gap between collaborative network models and spatial models by both extending the research in collaborative oligopoly network models \cite{goyal2003} and \cite{LichterGriffinFriesz2011}, by introducing the spatial transport network and by extending spatial oligopoly models \cite{Harker1986,dafermos1987,Miller1996,Nagurney2009}, and by introducing firm collaboration.  We have developed a generalized model using variational inequalities and shown in a subset of cases, we can analytically show that we may construct games that result in stable collaboration graphs with an arbitrary degree sequence.



\bibliographystyle{plain}
\bibliography{Biblio-Database2}
\end{document}